\documentclass[12pt]{article}
\usepackage{amsfonts}
\usepackage{amsmath}

\newtheorem{theorem}{Theorem}

\newtheorem{example}{Example}

\newtheorem{proposition}[theorem]{Proposition}
\newtheorem{remark}{Remark}

\newenvironment{proof}[1][Proof]{\noindent\textbf{#1.} }{\ \rule{0.5em}{0.5em}}

\begin{document}

\title{The limiting behavior of some infinitely divisible exponential
dispersion models}
\author{Shaul K. Bar-Lev\thanks{%
Department of Statistics, University of Haifa, Haifa 31905, Israel (email:
barlev@stat.Haifa.ac.il)} and G\'{e}rard Letac\thanks{%
Laboratoire de Statistique et Probabilit\'{e}s, Universit\'{e} Paul
Sabatier, 31062 Toulouse, France (letac@cict.fr) }}
\maketitle

\begin{abstract}
Consider an exponential dispersion model (EDM) generated by a probability $%
\mu $ on $[0,\infty )$ which is infinitely divisible with an unbounded L\'{e}%
vy measure $\nu $. The Jorgensen set (i.e., the dispersion parameter space)
is then $\mathbb{R}^{+}$, in which case the EDM is characterized by two
parameters: $\theta _{0}$ the natural parameter of the associated natural
exponential family and the Jorgensen (or dispersion) parameter $t$. Denote
by $EDM(\theta _{0},t)$ the corresponding distribution and let $Y_{t}$ is a
r.v. with distribution $EDM(\theta _{0},t)$. Then if $\nu ((x,\infty ))\sim
-\ell \log x$ around zero we prove that the limiting law $F_{0}$ of $%
Y_{t}^{-t}$ as $t\rightarrow 0$ is of a Pareto type (not depending on $%
\theta _{0}$) with the form $F_{0}(u)=0$ for $u<1$ and $1-u^{-\ell }$ for $%
u\geq 1$. Such a result enables an approximation of the distribution of $%
Y_{t}$ for relatively small values of the dispersion parameter of the
corresponding EDM.\medskip\ Illustrative examples are provided.\smallskip

\medskip

\end{abstract}

\section{Introduction}

Let $\left\{ F_{t}:0<t<t_{0}\leq \infty \right\} $ be a family of
distributions associated with positive r.v.'s $\left\{ Y_{t}\right\} $ and
Laplace transforms (LT's) $L_{t}(u)=E(e^{-uY_{t}})$. Also let $Y_{0}$ be a
r.v. with distribution $F_{0}$. Bar-Lev and Enis (1987, Theorem 1) showed
that $Y_{t}^{-t}\overset{D}{\rightarrow }Y_{0}$ as $t\rightarrow 0$ (where $%
\overset{D}{\rightarrow }$ designates a convergence in distribution) iff%
\begin{equation}
\lim_{t\rightarrow 0}L_{t}(u^{1/t})=\bar{F}_{0}(u)=1-F_{0}(u)  \label{1}
\end{equation}%
at all continuity points of $F_{0}$. As Bar-Lev and Enis indicated, such a
result can be viewed as a "centralization" problem in the following sense.
In many cases the limiting distribution of $Y_{t}$ as $t\rightarrow 0$ is
degenerate. A measurable transformation $g_{t}(Y_{t})$ is then sought whose
limiting distribution is non-degenerate. Accordingly, their Theorem 1
suggests a consideration of $g_{t}(Y_{t})=Y_{t}^{-t}$ (or, equivalently, of $%
-t\ln Y_{t}$) whose limiting distribution is non-degenerate (provided that (%
\ref{1}) is satisfied). Bar-Lev and Enis presented several examples which
satisfy (\ref{1}). However, these examples heavily depend on the explicit
(and relatively 'nice') form of $L_{t}$.

A natural question then arises: Can one delineate subclasses of
distributions which satisfy (\ref{1}), regardless of the explicit form of $%
L_{t}$? Indeed, in this note we provide such a subclass, namely a subclass
of exponential dispersion models (EDM's) generated by a probability $\mu $
on $[0,\infty )$ which is infinitely divisible of type $1$ (c.f., Jorgensen,
1987, 1997, 2006, and Letac and Mora, 1990). For such a subclass, the EDM is
characterized by two parameters: $\theta _{0}$ the natural parameter of the
associated natural exponential family (NEF) and the Jorgensen (or,
equivalently, the dispersion parameter) $t\in \mathbb{R}^{+}$. We denote
such a subclass of distributions by $EDM(\theta _{0},t)$ and prove that if $%
Y_{t}$ has a distribution in $EDM(\theta _{0},t)$ then $Y_{t}^{-t}\overset{D}%
{\rightarrow }Y_{0}$ as $t\rightarrow 0$, where the distribution $F_{0}$ of $%
Y_{0}$ is of a Pareto type (not depending on $\theta _{0}$) with the form%
\begin{equation*}
F_{0}(u)=\left\{ 
\begin{tabular}{l}
$0,\ \ \ \ \ \ \ \ \ \ $if $u<1,$ \\ 
$1-u^{-\ell },$ if $u\geq 1,$%
\end{tabular}%
\right.
\end{equation*}%
for some $\ell >0$. Such a result enables an approximation of the
distribution of $Y_{t}$ for relatively small values of the dispersion
parameter of the corresponding EDM.\medskip

This note is organized as follows. In Section 2 we first introduce some
preliminaries on NEF's and EDM's and then present our main result. Section 3
is devoted to some examples.

\section{Preliminaries and the main result}

We first briefly introduce some preliminaries related to NEF's and their
associated EDM's. Let $\mu $ be a probability measure on $\mathbb{R}$.
Assume that the effective domain of $\mu $ has a nonempty interior, i.e.,%
\begin{equation*}
\Theta \doteq \mathrm{int}\text{ }D=\left\{ \theta \in \mathbb{R}:L(\theta
)=\int_{\mathbb{R}}e^{-\theta x}\mu (dx)<\infty \right\} \neq \phi .
\end{equation*}%
The NEF generated by $\mu $ is then given by the set of probabilities%
\begin{equation*}
\left\{ P(\theta ,\mu )(dx)=\frac{e^{-\theta x}}{L(\theta )}\mu (dx),\
\theta \in D\right\} .
\end{equation*}%
Note that since $\mu $ is a probability measure then $0\in D$. The Jorgensen
set is defined by%
\begin{equation*}
\Lambda =\left\{ t\in \mathbb{R}^{+}:L^{t}\text{ is a LT of some measure }%
\mu _{t}\text{ on }\mathbb{R}\right\} ,
\end{equation*}%
whereas the corresponding EDM is the class of probabilities%
\begin{equation}
\left\{ P(\theta ,t,\mu _{t})(dx)=\frac{e^{-\theta x}}{L^{t}(\theta )}\mu
_{t}(dx),\theta \in D,t\in \Lambda \right\} ,  \label{2}
\end{equation}%
Note that the class of EDM's is abundant since any probability measure with
a LT generates an EDM. An EDM is therefore parameterized by the two
parameters $(\theta ,t)\in D\times \Lambda $, where $\theta $ is the natural
parameter of the corresponding NEF and $t$ is termed the dispersion
parameter. EDM's have a variety of applications in various areas, in
particular, in generalized linear models (replacing the normal model as the
error model distribution) and actuarial studies. Note that if $\mu $ is
infinitely divisible then the Jorgensen set (or, equivalently, the
dispersion parameter space) is $\Lambda =\mathbb{R}^{+}$. Also note that if $%
Y_{t}$ is a r.v. with distribution in (\ref{2}) then its LT is given by%
\begin{equation}
E(e^{-sY_{t}})=\left( \frac{L(\theta +s)}{L(\theta )}\right) ^{t}.  \label{3}
\end{equation}

Now we consider the case where $\mu $ is infinitely divisible law of type 1
concentrated on $\mathbb{R}^{+}.$ By this we mean that there exists an
unbounded positive measure $\nu $ on $(0,\infty )$ such that for $\theta
\geq 0$ one has $L(\theta )=\int_{0}^{\infty }e^{-\theta x}\mu (dx)\doteq
e^{k(\theta )}$ with $k(\theta )=-\int_{0}^{\infty }(1-e^{-\theta x})\nu
(dx) $ and $\int_{0}^{\infty }\min (1,x)\nu (dx)<\infty .$ Therefore $\nu $
is the L\'{e}vy measure of $\mu .$ The L\'{e}vy process associated with such
a $\mu $ is sometimes called a pure jump subordinator (in this respect, of L%
\'{e}vy measures for NEF's, see also Kokonendji and Khoudar, 2006). Note
that this implies that $\lim_{\theta \rightarrow \infty }k(\theta )=-\infty $
since $\nu $ is unbounded and therefore $\lim_{\theta \rightarrow \infty
}L(\theta )=0$ and $\mu (\{0\})=0.$ We are now ready to present our main
result relating to the limiting distribution of $Y_{t}^{-t}$ as $%
t\rightarrow 0$.

\begin{proposition}
\vspace{4mm}Assume that $\mu $ is an infinitely divisible probability
measure of type $1$.\ Also assume that $G(x)\doteq \int_{x+}^{\infty }\nu
(dy)=\nu ((x,\infty ))$ is such that 
\begin{equation*}
\lim_{x\rightarrow 0}G(x)/\log x=-\ell
\end{equation*}%
for some $\ell >0.$ Let $\theta _{0}\geq 0$, then 
\begin{equation}
\lim_{t\rightarrow 0}\left( \frac{L(\theta _{0}+u^{1/t})}{L(\theta _{0})}%
\right) ^{t}=\left\{ 
\begin{tabular}{l}
$1,\ \ \ \ \ \ \ \ \ \ $if $u<1,$ \\ 
$u^{-\ell },$ if $u\geq 1,$%
\end{tabular}%
\right.  \label{4}
\end{equation}%
implying by (\ref{1}) that $Y_{t}^{-t}\overset{D}{\rightarrow }Y_{0}$ as $%
t\rightarrow 0$, where the distribution $F_{0}$ of $Y_{0}$ is given by 
\begin{equation*}
F_{0}(u)=\left\{ 
\begin{tabular}{l}
$0,\ \ \ \ \ \ \ \ \ \ $if $u<1,$ \\ 
$1-u^{-\ell },$ if $u\geq 1,$%
\end{tabular}%
\right.
\end{equation*}
\end{proposition}

\begin{proof}
We first prove (\ref{4}) for $\theta _{0}=0.$ For this, we give another
presentation of $k$ in terms of $G$ which is obtained by an integration by
parts. For $\epsilon >0$ consider the Stieltjes integral 
\begin{equation}
k_{\epsilon }(\theta )=-\int_{\epsilon +}^{\infty }(1-e^{-\theta x})\nu
(dx)=(e^{-\theta \epsilon }-1)G(\epsilon )-\theta \int_{\epsilon }^{\infty
}e^{-\theta x}G(x)dx.  \label{KE}
\end{equation}%
Since for $\epsilon \rightarrow 0$ we have $(e^{-\theta \epsilon }-1)\sim
-\theta \epsilon $ and $G(\epsilon )\sim -\ell \log \epsilon $ we get 
\begin{equation}
k(\theta )=\lim_{\epsilon \rightarrow 0}k_{\epsilon }(\theta )=-\theta
\int_{0}^{\infty }e^{-\theta x}G(x)dx.  \label{KG}
\end{equation}%
If $0<u<1$ we have $\lim_{t\rightarrow 0}u^{1/t}=0.$ Since $\mu $ is a
probability $L(0)=1$ and thus $\lim_{t\rightarrow 0}L(u^{1/t})^{t}=1.$ If $%
u=1$ we also have that $\lim_{t\rightarrow 0}L(1)^{t}=1.$ If $u>1$, We fix
an arbitrary $0<\epsilon <\ell $. By the definition of $\ell $ there exists $%
0<\eta <1$ such that if $0<x<\eta $ then $-(\ell -\epsilon )\log
x<G(x)<-(\ell +\epsilon )\log x.$ We now use (\ref{KG}) for writing 
\begin{equation}
\left\vert k(\theta )+\theta \int_{\eta }^{\infty }e^{-\theta x}G(x)dx-\ell
\theta \int_{0}^{\eta }e^{-\theta x}\log xdx\right\vert <-\epsilon \theta
\int_{0}^{\eta }e^{-\theta x}\log xdx  \label{ENC}
\end{equation}%
and observing that 
\begin{equation}
\theta \int_{\eta }^{\infty }e^{-\theta x}G(x)dx\leq G(\eta )\theta
\int_{\eta }^{\infty }e^{-\theta x}dx=G(\eta )e^{-\eta \theta }\underset{%
\theta \rightarrow \infty }{\rightarrow }0.  \label{EG}
\end{equation}%
We need now the following evaluation. For $\eta >0$ we have 
\begin{equation}
\lim_{\theta \rightarrow \infty }\frac{1}{\log \theta }\theta \int_{0}^{\eta
}e^{-\theta x}\log xdx=-1.  \label{EAS}
\end{equation}%
To prove (\ref{EAS}), we obtain, by a change of variable to $v=\theta x$,
that 
\begin{equation}
\frac{1}{\log \theta }\theta \int_{0}^{\eta }e^{-\theta x}\log xdx=\frac{1}{%
\log \theta }\int_{0}^{\eta \theta }e^{-v}\log vdv-\int_{0}^{\eta \theta
}e^{-v}dv\rightarrow _{\theta \rightarrow \infty }=0-1,  \label{5}
\end{equation}%
where the last term on right hand side of (\ref{5}) follows since $%
\int_{0}^{\infty }e^{-v}\log vdv$ converges and $\int_{0}^{\infty
}e^{-v}dv=1.$ We now divide both sides of (\ref{ENC}) by $\log \theta $ and
let $\theta \rightarrow \infty .$ From (\ref{EG}) and (\ref{EAS}) we get
that for all $\epsilon >0$ 
\begin{equation*}
-\ell -\epsilon \ell \leq \liminf_{\theta \rightarrow \infty }\frac{1}{\log
\theta }k(\theta )\leq \limsup_{\theta \rightarrow \infty }\frac{1}{\log
\theta }k(\theta )\leq -\ell +\epsilon \ell ,
\end{equation*}%
i.e., $\lim_{\theta \rightarrow \infty }\frac{1}{\log \theta }k(\theta
)=-\ell .$ Applying this to $\theta =u^{1/t}$ with a fixed $u>1$ and letting 
$t\rightarrow 0$, we get 
\begin{equation*}
\lim_{t\rightarrow 0}tk(u^{1/t})=-\ell \log u,\ \ \lim_{t\rightarrow
0}(L(u^{1/t}))^{t}=\frac{1}{u^{\ell }},
\end{equation*}%
which is the desired result. Finally suppose that $\theta _{0}>0$ and denote 
$k_{\theta _{0}}(\theta )=k(\theta _{0}+\theta )-k(\theta _{0}).$ Trivially, 
\begin{equation*}
k_{\theta _{0}}(\theta )=-\int_{0}^{\infty }(1-e^{-\theta x})\nu _{\theta
_{0}}(dx)\text{ with }\nu _{\theta _{0}}(dx)=e^{-\theta _{0}x}\nu (dx)
\end{equation*}%
and consider 
\begin{eqnarray}
G_{\theta _{0}}(x) &=&\nu _{\theta _{0}}((x,\infty ))=\int_{x+}^{\infty
}e^{-\theta _{0}y}\nu (dy)  \notag \\
&=&e^{-\theta _{0}x}G(x)-\theta _{0}\int_{x}^{\infty }e^{-\theta _{0}y}G(y)dy
\label{IP2} \\
&=&k_{x}(\theta _{0})+G(x)  \label{KE2}
\end{eqnarray}%
Line (\ref{IP2}) is obtained by an integration by parts, whereas line (\ref%
{KE2}) uses (\ref{KE}). We see easily from (\ref{KE2}) that $%
\lim_{x\rightarrow 0}G_{\theta _{0}}(x)/\log x=-\ell .$ Therefore we are in
the same situation as in the proof of the first part with $\theta _{0}=0$,
and thus the proof is completed.\medskip
\end{proof}

The following remark is useful to obtain a genralization of the three
examples presented in Section 3 for any $\ell >0$.

\begin{remark}
Suppose that $(\mu _{t})_{t>0}$ is a family of infinitely divisible
distributions with $\mu _{t}\ast \mu _{s}=\mu _{t+s},$ where $t,s>0.$ Assume
that $\mu _{1}$ fulfills the premises of Proposition 1 with $G(x)\sim -\ell
\log x.$ Then obviously for any fixed $t>0$, $\mu _{t}$ also fulfills such
premises with $t\ell $ replacing the role of $\ell $. In all of the examples
below, we have $\ell =1$ and this remark shows how to get from them other
examples with arbitrary $\ell >0.$
\end{remark}

\section{\protect\vspace{4mm}\noindent \textbf{Examples}}

\begin{example}
If $\mu (dx)=e^{-x}1_{(0,\infty )}(x)dx$ the corresponding infinitely
divisible family is the gamma family with scale parameter 1. The L\'{e}vy
measure here is $\nu (dx)=e^{-x}1_{(0,\infty )}(x)\frac{dx}{x}$ and $\ell
=1. $
\end{example}

\begin{example}
A discrete example is 
\begin{equation*}
\nu (dx)=\sum_{n=1}^{\infty }\delta _{1/n}.
\end{equation*}%
We have $k(\theta )=\sum_{n=1}^{\infty }\frac{1}{n}(1-e^{-\theta /n})$, $%
G(x)=\sum_{n=1}^{[1/x]}1/n\sim -\log x$ if $x\rightarrow 0$ and $\ell =1.$%
The probability $\mu $ is the distribution of $\sum_{n=1}^{\infty }\frac{%
X_{n}}{n},$ where $X_{n}$ is Poisson distributed with mean $1/n$ and the $%
(X_{n})_{n=1}^{\infty }$ are independent.
\end{example}

\begin{example}
Utilizing Example 2.2 in \cite{BE1987}, consider the infinitely divisible
distribution $\mu $ on $(0,\infty )$ with Laplace transform defined on $%
\theta \geq 0$ given by $\theta +1-\sqrt{\theta ^{2}+2\theta }.$ Note that
the densities of the corresponding EDM are Bessel densities related to a
symmetric random walk (see Feller, 1971, pp. 60-61).
\end{example}

The related L\'{e}vy measure of $\mu $ is 
\begin{equation}
\nu (dx)=\ _{1}F_{1}(1/2;1;-2x)\mathbf{1}_{(0,\infty )}(x)\frac{dx}{x},
\label{KM}
\end{equation}%
where $_{1}F_{1}(a;b;z)$ is the so called confluent entire function defined
by $\sum_{n=0}^{\infty }\frac{(a)_{n}}{n!(b)_{n}}z^{n},$ where $(a)_0=1$ and 
$\left( a\right) _{n+1}=(a+n)(a)_n$ define the Pochhammer symbol $(a)_n.$
Therefore $\nu $ has a density equivalent to $1/x$ when $x\rightarrow 0$
which implies that $G(x)\sim -\log x.$ Proposition 1 is thus satisfied with $%
\ell =1.$ To check the correctness of (\ref{KM}) observe that if $k(\theta
)=\log (\theta +1-\sqrt{\theta ^{2}+2\theta })$ then 
\begin{equation*}
k^{\prime }(\theta )=-\int_{0}^{\infty }e^{-\theta x}x\nu (dx)=-\frac{1}{%
\sqrt{\theta +2}}\frac{1}{\sqrt{\theta }}=-\int_{0}^{\infty }e^{-\theta
x}f(x)dx\times \int_{0}^{\infty }e^{-\theta y}g(y)dy,
\end{equation*}%
where 
\begin{equation*}
f(x)=e^{-2x}\frac{1}{\sqrt{\pi x}}\mathbf{1}_{(0,\infty )}(x),\ g(y)=\frac{1%
}{\sqrt{\pi y}}\mathbf{1}_{(0,\infty )}(y).
\end{equation*}%
Therefore the density of $x\nu (dx),$ $x>0$, is given by the convolution
product $f\ast g$, where by a change of variable $y=tx$ and employing a
Taylor expansion, one gets 
\begin{equation*}
\int_{0}^{x}f(y)g(x-y)dy=\frac{1}{\pi }\int_{0}^{1}\frac{e^{-2xt}}{\sqrt{%
t(1-t)}}dt=\ _{1}F_{1}(1/2;1;-2x).
\end{equation*}
Let us fix $t>0.$ Recall (see \cite{Feller1971}) that for $\theta>0$ the function $(\theta +1-\sqrt{\theta ^{2}+2\theta })^t$ is the Laplace transform of the density $f_t(x)=\frac{t}{x}e^{-x}I_t(x)\textbf{1}_{(0,\infty)}(x)$ where the Bessel function $I_t(x)$ is $$I_t(x)=\sum_{n=0}^{\infty}\frac{1}{n!\Gamma(n+1+t)}\frac{x^{2n+t}}{2^{n+t}}.$$ We fix now $\ell>0$. If $Y_t$ has density $f_{t\ell}$ this implies that the density of $U=Y_t^{-t}$
is $$g_t(u)=\ell t^2u^{2t+1}e^{-\frac{1}{u^t}}I_{t\ell}(\frac{1}{u^t})\textbf{1}_{(0,\infty)}(u).$$ It would be quite delicate to prove directly from this last formula that when $t\rightarrow 0$ the law $g_t(u)du$ converges  to the Pareto law $$\textbf{1}_{(1,\infty)}(u)\frac{\ell du}{u^{\ell+1}}$$ as shown by our Proposition.

\end{document}